\documentclass[a4paper,11pt]{amsart}
\usepackage{graphicx} 
\usepackage{amsfonts, amssymb, amsmath}
\usepackage{amsthm}
\usepackage{mathtools}
\usepackage[margin=1in]{geometry}
\usepackage{enumerate}
\usepackage[all]{xy}
\usepackage{color, xcolor}

\usepackage{tikz, tikz-cd}
\usetikzlibrary{3d}
\usepackage{caption}
\usepackage{subcaption}
\usepackage[all]{xy}

\usepackage{listings}

\newcommand{\PP}{\mathbb{P}}

\newcommand{\QQ}{\mathbb{Q}}

\newcommand{\cX}{\mathcal{X}}

\newcommand{\n}{\noindent}

\theoremstyle{definition}
\newtheorem*{Def}{Definition}
\newtheorem{Thm}{Theorem}[section]
\newtheorem{Prop}[Thm]{Proposition}
\newtheorem{Lem}[Thm]{Lemma}
\newtheorem{Rmk}[Thm]{Remark}

\newtheorem{Conj}[Thm]{Conjecture}
\newtheorem{Ex}[Thm]{Example}

\usepackage{epic}

\title{Weighted Projective Spaces Admitting $\QQ$-Gorenstein Smoothings to $\PP^3$}
\author{Jungkai Alfred Chen and Yongnam Lee}
\date{September 2026}

\address{Department of Mathematics, National Taiwan University, No. 1, Sec. 4, Roosevelt Rd., Taipei 10617, Taiwan}
\email{jkchen@ntu.edu.tw}

\address{Center for Complex Geometry, Institute for Basic Science (IBS), 55 Expo-ro, Yuseong-gu, Daejeon 34126, Korea}
\email{ynlee@ibs.re.kr}


\begin{document}

\begin{abstract}
We study well-formed weighted projective threefolds that admit $\QQ$-Gorenstein smoothings to $\PP^3$. Two families are known: the $\PP^2$-type and the $Q$-type, and it is conjectured that these are the only possibilities. We derive numerical and local necessary conditions for such a smoothing. In addition to the anticanonical volume equation, constancy of the anticanonical Hilbert polynomial yields a further identity when all codimension two singularities are of $A$-type. We also obtain a semigroup condition governing the existence of global smoothing directions along codimension two curves with transverse $A$-type singularities. We apply these conditions to prove the expected classification in several cases. In particular, for every fixed square-free integer $d$, there are only finitely many $\QQ$-Gorenstein smoothable spaces $\PP(1,a,b,c)$ such that $\gcd(a,b)=d$. Our method reduces the possible weights to a finite exact computation; for every prime $p\le100$, the computation produces only members of the two expected families. Finally, we prove the classification when $\gcd(a,b)=d$ and $a=d^2$. 
\end{abstract}

\maketitle

\section{Introduction}

In this paper, we work over the complex numbers. Weighted projective spaces arise naturally as toric degenerations of projective varieties and provide explicit testing grounds for questions concerning $\QQ$-Gorenstein smoothings. In this paper we consider the following classification problem:
Which well-formed 3-dimensional weighted projective spaces
$X=\PP(w_0,w_1,w_2,w_3)$ admit a $\QQ$-Gorenstein smoothing whose general fibre is $\PP^3$?

Two infinite families are known. The first consists of the weighted projective spaces
\[\PP(a^2,b^2,c^2,abc),
\qquad a^2+b^2+c^2=3abc.\]
We call these weighted projective spaces of $\PP^2$-type. The second consists of
\[\PP(2a^2,b^2,c^2,4abc),
\qquad 2a^2+b^2+c^2=4abc,\]
which we call $Q$-type. It is conjectured by DeVleming \cite{DeV22} that, up to reordering and normalization of the weights, these are the only weighted projective threefolds admitting a $\QQ$-Gorenstein smoothing to $\PP^3$. These weighted projective spaces are related to $\QQ$-Gorenstein degeneration of $\PP^2$ and of a smooth quadric surface
\cite{HP10}. These relations are explained in \cite[Section 3]{DLT} using the cone.

Our purpose is to develop numerical and local tools for this classification, including smoothing section condition on codimension two coordinate curves, and to prove the conjecture in several cases. Let $S=w_0+w_1+w_2+w_3.$

Invariance of the anticanonical volume gives the following equation.
\[S^3=64w_0w_1w_2w_3. \tag{V}\] 
We call this {\it the volume equation}. The volume equation alone is far from sufficient.
Our first additional condition comes from comparing the linear coefficients of the anticanonical Hilbert polynomials. Let
$g_{ij}=\gcd(w_i,w_j),$
and let $w_k,w_\ell$ be the complementary weights. When every codimension-two singularity is of $A$-type, we prove
\[\sum_{i<j}(w_i-w_j)^2
=8\sum_{i<j}(g_{ij}^2-1)w_kw_\ell.
\tag{H}\]
We call (H) {\it the Hilbert identity}. The left-hand side measures the discrepancy among the four weights, while the terms on the right arise from the double poles of the Hilbert series associated with the singular coordinate curves.

The $A$-type congruence is still not sufficient for smoothability. Suppose that
$C_{ij}=\{x_k=x_\ell=0\}$ is a coordinate curve with generic transverse singularity $A_{g_{ij}-1}$. We show that smoothing this curve requires the stronger condition:
$w_k+w_\ell=g_{ij}\delta$ for some positive integer $\delta$, and 
\[g_{ij}\delta\in\Big\langle \frac{w_i}{g_{ij}},\ \frac{w_j}{g_{ij}}\Big\rangle_{\mathbb Z\ge 0}
\quad\text{or}\quad
(g_{ij}-1)\delta\in\Big\langle \frac{w_i}{g_{ij}}, \ \frac{w_j}{g_{ij}}\Big\rangle_{\mathbb Z\ge 0}.\tag{S}\]
This is a global section condition on the first-order smoothing parameter. We call (S) {\it the smoothing-section condition}. 

Our first classification result concerns a weight system containing \(1\) and a pair with fixed prime gcd. For instance, for $X=\PP(1,a,b,c)$ with $\gcd(a,b)=2$ after reordering $a,b,c$, if $X$ admits a $\QQ$-Gorenstein smoothing to $\PP^3$, then
\[X\simeq
\PP(1,1,2,4),\quad
\PP(1,2,9,12),\quad\text{or}\quad
\PP(1,4,10,25).\]
In particular, $X$ is of $\PP^2$-type or $Q$-type.

The proof combines the smoothing-section condition with the volume equation. It reduces the weights to a finite exact calculation.
More generally, for every fixed square-free integer $d$, we prove that there are only finitely many possibilities for a $\QQ$-Gorenstein smoothable space $\PP(1,a,b,c)$ containing a pair of weights with $\gcd(a,b)=d$. Exact calculations for prime $p\le100$ produce only members of the two expected families.

Our main arithmetic classification result is the following.

\begin{Thm}\label{mainA}
Let
\[X=\PP(1,a,b,c)\]be a well-formed weighted projective space. Suppose that 
$\gcd(a,b)=d$ and $a=d^2$. Assume that every codimension-two singularity of $X$ is of $A$-type. If $X$ admits a $\QQ$-Gorenstein smoothing to $\PP^3$, then $X$ is of $\PP^2$-type or $Q$-type.
\end{Thm}

In Theorem~\ref{mainA}, we do not assume that $d$ is square free. But we assume that every codimension-two singularity of $X$ is $A$-type. The proof reduces the Hilbert identity to
\[b_0^2-4mb_0+6m^2-2c-e^2=0,\]
where $b=db_0$, $b_0c=m^3$, and $e=\gcd(b_0,c)$. A valuation argument gives $e\mid m$, after which the problem reduces to the cubic Thue equations $F(U,V)=1$ or $F(U,V)=2$ for
\[F(U,V)=U^3-4U^2V+6UV^2-2V^3.\]
The resulting three arithmetic branches give the $\PP^2$-type family, the $Q$-type family, and one additional numerical branch. The last branch is excluded by the smoothing-section condition.

The paper is organized as follows. In Section~2 we derive the
anticanonical Hilbert identity and the smoothing-section
condition. In Section~3 we prove finiteness for fixed square-free
gcd, classify the gcd-$2$ case, and record explicit computations
for small prime gcd. Section~4 is devoted to the main theorem; its
proof uses the volume equation, the Hilbert identity, smoothing-section condition, $p$-adic valuations, and explicit
Thue equations.

\subsection*{Acknowledgements}
J. Chen is partially supported by National Center for Theoretical Sciences and National Science and Technology Council (115-2811-M-002-126) of Taiwan.  Y. Lee is supported by the Institute for Basic Science
(IBS-R032-D1). Y. Lee would like to thank the Department of Mathematics and NCTS at National Taiwan University for the hospitality during his visit in 2026.  The authors would like to thank Kristin DeVleming for some useful discussion and comments.

\section{Hilbert identity and smoothing-section condition}

Let $X=\mathbb{P}(w_0, w_1, w_2, w_3)$ be a well-formed weighted projective space. Since it is well-formed, we have $\operatorname{gcd}(w_i, w_j, w_k)=1$ for distinct $\{i,j,k\} \subset \{0,1,2,3\}$. Details on weighted projective spaces can be found in \cite{BR86, Dol82, IF00}.
In this paper, we will always assume well-formedness. We assume that $X$ is a $\QQ$-Gorenstein degeneration of $\mathbb{P}^3$,  that is, there exists a flat projective  morphism $\pi: \cX\to\Delta$ over a DVR with the central fiber $X$ such that $\pi^{-1}(t)=X_t\cong \PP^3$ for $t\ne 0$ and $K_{\cX/\Delta}$ is $\QQ$-Cartier.

\medskip

{\bf Notation and Convention.} In the sequel, we fix the index $\{i,j,k, \ell\}=\{0,1,2,3\}$ and the index $\{i,j,k\}=\{1,2,3\}$ if $w_0$ is assumed to be $1$. 

Given two positive integers $p$ and $q$, the set $\langle p,q \rangle_{\mathbb{Z} \geq0}$ or simply $\langle p,q \rangle$ stands for the set of integers $\{ \alpha p +\beta q~|~ \alpha, \beta \in \mathbb{Z}_{\ge 0}\}$. 

Let $S=\sum_{i=0}^3w_i$ and $L=\operatorname{lcm}(w_0,w_1,w_2,w_3)$. 
For a well-formed weighted projective space,
\(\operatorname{Cl}(X)\cong \mathbb Z\langle\mathcal O_X(1)\rangle,\) while the Cartier classes form the subgroup
$\operatorname{Pic}(X)
=\mathbb Z\langle\mathcal O_X(L)\rangle
\subset\operatorname{Cl}(X).$ Consequently,
\[\mathcal O_X(d)\ \text{is Cartier}
\quad \text{if and only if} \quad
L\mid d.\]
Therefore \(\mathcal O_X(-rK_X)=\mathcal O_X(rS)\) is Cartier exactly when
\[\mathcal O_X(rS)\ \text{is Cartier}
\quad \text{if and only if} \quad
L\mid rS.\]
The smallest positive integer satisfying this divisibility is $r=\frac{L}{\gcd(L,S)}.$
Therefore,  the Cartier index of \(K_X\) is $r=L/\gcd(S,L)$. 

Let $H_X(m)=\chi\!\left(X,\mathcal O_X(-mrK_X)\right)$. This is  the Hilbert polynomial associated  with the anticanonical $\QQ$-line bundle on $X$ with Cartier index $r$.
Then $H_X(m)$ satisfies the following Hilbert identity. This constancy of the coefficients of $H_X(m)$ coming from the anticanonical $\QQ$-line bundle gives the fundamental equation.

\begin{Lem}\label{Hilbert identity} Let $X=\PP(w_0,w_1,w_2,w_3)$ be a well-formed weighted projective space admitting a $\QQ$-Gorenstein smoothing to $\PP^3$. Then   
\begin{enumerate}[(i)]
\item Anticanonical volume is $64$, that is, 
    \[S^3=(w_0+w_1+w_2+w_3)^3=64w_0w_1w_2w_3.\]
  \item  \[ [t^{mrS}]
\frac{1}{(1-t^{w_0})(1-t^{w_1})(1-t^{w_2})(1-t^{w_3})}
=\binom{4mr+3}{3} \quad \text{for every $m\geq 0$}.\]
 Here the notation $[t^n]H(t)$ means “take the coefficient of $t^n$ in the power-series expansion of $H(t)$”.
\end{enumerate}
\end{Lem}

\begin{proof}
A sufficiently divisible multiple of the relative anticanonical divisor is a line bundle on the whole $\QQ$-Gorenstein family. Its Euler characteristic is constant in a flat family, so
\[H_X(m):=\chi\left(X,\mathcal O_X(-mrK_X)\right)=
\chi\left(\PP^3,\mathcal O_{\PP^3}(4mr)\right)=:H_{\PP^3}(m).\]
Both $H_X(m)$ and $H_{\PP^3}(m)$ are polynomials in $m$ of degree at most $3$. The preceding equality implies that the polynomial
$H_X(m)-H_{\PP^3}(m)$ has infinitely many roots  so it must be identically zero. Therefore
$H_X(m)=H_{\PP^3}(m)$ for every $m \ge 0$. 
It is easy to see that the LHS of (ii) computes $H_X(m)$ and the RHS of (ii) computes $H_{\PP^3}(m)$. 
\end{proof}

\begin{Rmk}
The cubic and quadratic coefficients of the anticanonical Hilbert polynomial are equivalent.
For a threefold, Riemann–Roch begins
\[\chi(X,\mathcal O_X(mD))=
\frac{D^3}{6}m^3
-\frac{K_XD^2}{4}m^2+\cdots.\]
Taking $D=-rK_X$ makes both coefficients multiples of $(-K_X)^3$:
\[D^3=r^3(-K_X)^3,\qquad
-K_XD^2=r^2(-K_X)^3.\]
Therefore both coefficients contain the same numerical invariant, the anticanonical volume.
\end{Rmk}

For every singular edge $(w_i,w_j)$, put $g_{ij}:=\gcd(w_i,w_j)$, with complementary weights $w_k,w_\ell$. The next lemma shows some necessary condition when $X=\mathbb{P}(w_0, w_1, w_2, w_3)$ admits a smoothing to $\PP^3$.

\begin{Lem}\label{smoothing}
\begin{enumerate}[(i)]
\item If $w_i >1$, then $g_{ij} >1$ for some $j\ne i$.
\item If $d:=g_{ij} >1$ and $d$ is square-free then $X$ has only $A_{d-1}$ singularities along general points of the curve \(C_{ij}=\{x_k=x_\ell=0\}\subset X\).
\item Suppose that the generic transverse singularity of $X$ along $C_{ij}$ is $A_{d-1}$. Then 
\(w_k+w_\ell= d \delta \) for some positive integer \(\delta\), and 
\[d \delta\in\Big\langle \frac{w_i}{d},\ \frac{w_j}{d} \Big\rangle_{\mathbb Z_{\ge0}}
\quad\text{or}\quad
(d-1)\delta\in\Big\langle \frac{w_i}{d}, \ \frac{w_j}{d} \Big\rangle_{\mathbb Z_{\ge0}}. \]

In particular, if $d=2$, then $w_k+w_\ell=2\delta\in\Big\langle \frac{w_i}{2},\ \frac{w_j}{2}\Big\rangle_{\mathbb Z_{\ge0}}$.
\end{enumerate}
\end{Lem}

\begin{proof}
(i) If $g_{ij}=1$ for all $j\ne i$ then the point $P_i$ is an isolated quotient singularity.  But a three-dimensional isolated quotient singularity is rigid \cite{Sch}, so $X$ cannot admit a smoothing.

(ii) is proved in \cite[Corollary 2.21]{DLT}.

(iii) Let $C_{ij}=\{x_k=x_\ell=0\}\simeq\PP(b_i,b_j), \ 
b_i=\frac{w_i}{d},\ b_j=\frac{w_j}{d},$ where $d=\gcd(w_i,w_j)>1$.

Because the singularity is of type $A_{d-1}$,
$w_k+w_\ell=d \delta$ for some positive integer $\delta$:

After normalizing the $\mu_d$-action, let $x,y$ be coordinates with action
$(x,y)\longmapsto(\zeta x,\zeta^{-1}y).$ The invariant coordinates are
\[u=x^d,\qquad v=xy,\qquad w=y^d,\]
and satisfy $uw=v^d.$ Then the degrees along $C_{ij}$ are
\[\deg u=w_k,\qquad
\deg w=w_\ell,\qquad
\deg v=\delta.\]
Thus both sides of $uw=v^d$ have degree
$w_k+w_\ell=d \delta.$

A general deformation of $A_{d-1}$ can be written, after eliminating the $v^{d-1}$-term, as
\[uw-\left(v^d+s_{d-2}v^{d-2}+\cdots+s_1v+s_0\right)=0.\]
Homogeneity requires $s_j\in H^0(C_{ij},\mathcal O_{C_{ij}}((d-j)\delta))$.

If both $s_0=0$ and $s_1=0,$ then the polynomial on the right is divisible by $v^2$.
Consequently, the point $u=w=v=0$ remains singular. Thus smoothing the generic transverse singularity requires
$s_0\ne0 \ \text{or}\ s_1\ne0.$
If $s_0\ne0$, then $H^0(C_{ij},\mathcal O(d\delta)) \ne 0$. Since  $C_{ij}\simeq\PP(b_i,b_j)$, this is equivalent to
$d\delta\in\langle b_i,b_j\rangle$. If $s_1\ne0$, then $H^0 (C_{ij},\mathcal O((d-1)\delta))\ne0$, equivalently,
$(d-1)\delta\in\langle b_i,b_j\rangle$.

For $d=2$, the miniversal normal form contains only the constant parameter: $uw-v^2+s_0=0.$ Therefore
$2\delta=w_k+w_\ell\in\langle b_i,b_j\rangle$. 
\end{proof}

\begin{Def}
We say that $X=\mathbb{P}(w_0,w_1,w_2,w_3)$ has $A$-type  singularities in codimension two if 
 $X$ has $A_{g_{ij}-1}$-singularities along the general points of the corresponding coordinate curve $C_{ij}$ whenever $g_{ij} >1$.

This is equivalent to $w_k \equiv-w_{\ell}\pmod {g_{ij}}.$ Thus, $A$-type in codimension two means:
\[g_{ij}\mid w_k+w_\ell
\quad\text{for every pair }i \ne j.\]
\end{Def}

\begin{Ex}
    The following two known families admit $\QQ$-Gorenstein smoothings to $\PP^3$. These two series satisfy the conditions in Lemma~\ref{Hilbert identity} and Lemma~\ref{smoothing}.
    \begin{enumerate} 
    \item $(a^2, b^2, c^2, abc)$, where $a^2+b^2+c^2=3abc$;
    \item $(2a^2, b^2, c^2, 4abc)$, where $2a^2+b^2+c^2=4abc$;
    \end{enumerate} 
    The first one is called $\PP^2$-type, and the second one is called $Q$-type. 
      It is straightforward to check that if a weighted projective space $X$ is $\PP^2$-type or $Q$-type then  every codimension-two singularity is $A$-type in $X$.

\end{Ex}

\begin{Conj} (\cite[Conjecture 4.64]{DeV22}) Every weighted projective space admitting $\QQ$-Gorenstein smoothings to $\PP^3$ is of $\PP^2$-type or $Q$-type. 
    
\end{Conj}

\begin{Lem}\label{linear} Let $X=\PP(w_0,w_1,w_2,w_3)$ be a well-formed weighted projective space admitting a $\QQ$-Gorenstein smoothing to $\PP^3$.  
Suppose every codimension-two singularity is $A$-type in $X$. Then invariance of the anticanonical Hilbert polynomial gives
\[\sum_{i<j}(w_i-w_j)^2=
8\sum_{i<j}(g_{ij}^2-1)w_kw_\ell. \tag{H}\]
\end{Lem}

\begin{proof}
It follows by comparing the linear coefficients of the anticanonical Hilbert polynomials.
Let
\[S=\sum_{i=0}^3w_i,\qquad
P=\prod_{i=0}^3w_i,\qquad
E_2=\sum_{i<j}w_iw_j.\]
Choose \(r>0\) so that \(rS\) is divisible by $L$ the lcm of $w_i$s. Then
$-rK_X=\mathcal O_X(rS)$ is Cartier, and by Lemma \ref{Hilbert identity}, 
\[H_X(m)=[t^{mrS}]\frac1{\prod_{i=0}^3(1-t^{w_i})}=
\binom{4rm+3}{3}.\]

We aim to extract the coefficient by using the Residue Theorem 
\[ [t^N]F(t)= \frac{1}{2 \pi i} \int F(t) t^{-N-1}dt.\]
Then we will compare the coefficient of the term linear in $m$. 

We consider $F(t)=\prod_{i=0}^3 \frac{1}{(1-t^{w_i})}$. Note that $F(t)$ has a pole or order $4$ at $t=1$ and some poles of order $2$ when $t$ is $d$-th root of unity where $d=\gcd(w_i,w_j)>1$. There is no pole of order $3$ since $\gcd(w_i,w_j, w_k)=1$. 

To compute the integral around $t=1$, we consider the transformation $t=e^{-z}$. Then 
\[ \frac{1}{2 \pi i} \int F(t) t^{-N-1}dt =  \frac{1}{2 \pi i} \int F(e^{-z}) e^{Nz}dz \]

Since
\[\frac1{1-e^{-wz}}=\frac1{wz}+\frac12+\frac{wz}{12}+O(z^3).\]
We have
\[\begin{array}{ll}
F(e^{-z})&=\frac1{Pz^4}\prod_{i=0}^3
\left(1+\frac{w_i z}{2}+\frac{w_i^2z^2}{12}+O(z^3)\right) \\
&=\frac{1}{Pz^4} \prod \left(1+\frac{S}{2}z+ \left( \frac1{12}\sum_iw_i^2+\frac14\sum_{i<j}w_iw_j \right)z^2+O(z^3)\right). \end{array}\]
Since $\sum_iw_i^2=S^2-2E_2,$ the coefficient of $z^2$ becomes
$\frac{S^2-2E_2}{12}+\frac{E_2}{4}
=\frac{S^2+E_2}{12}.$
Therefore
\[F(e^{-z})=\frac1P\left(z^{-4}
+\frac{S}{2} z^{-3}+\frac{S^2+E_2}{12} z^{-2}+c_{-1}+O(1)
\right), \] for some constant $c_{-1}$. 
Hence the residue of $F(e^{-z})e^{Nz}$ at $z=0$ is given by \[\frac{N^3}{6P}+\frac{SN^2}{4P}+\frac{(S^2+E_2)N}{12P}+c_{-1}.  \tag{2.1}\]

We next compute directly the contributions from the nontrivial double poles.
Fix a pair $i,j$, and let $d=g_{ij}>1$. At every nontrivial $d$-th root of unity $\zeta$, the Hilbert series has a double pole coming from
$(1-t^{w_i})(1-t^{w_j}).$ Its contribution to the coefficient of $N$ is
\[\frac{N}{w_iw_j}\frac{\zeta^{-N}}{(1-\zeta^{w_k})(1-\zeta^{w_\ell})}.\] 
Because $N=mrS$ is divisible by $L$, and the order of every pole divides $L$, we have $\zeta^N=1$.
By Lemma~\ref{smoothing},
$w_k+w_\ell\equiv0\pmod d.$ Also, well-formdness gives
$\gcd(w_k,d)=\gcd(w_\ell,d)=1.$

Write every nontrivial \(d\)-th root of unity as
$\eta_\lambda=e^{2\pi i\lambda/d}, \ \ \lambda=1,\dots,d-1.$ Then
$(1-\eta_\lambda)(1-\eta_\lambda^{-1})=2-2\cos\left(\frac{2\pi \lambda}{d}\right)=4\sin^2\left(\frac{\pi \lambda}{d}\right)$.
Therefore
$\frac1{(1-\eta_\lambda)(1-\eta_\lambda^{-1})}=
\frac14\csc^2\left(\frac{\pi \lambda}{d}\right)$.

Consequently,
\[ \sum_{\substack{\zeta^d=1\\\zeta\ne1}}
\frac1{(1-\zeta^{w_k})(1-\zeta^{w_\ell})}=
\sum_{\substack{\eta^d=1\\\eta\ne1}}
\frac1{(1-\eta)(1-\eta^{-1})}
=\frac{d^2-1}{12}. \]
The final equality follows from $\sum_{\lambda=1}^{d-1}\csc^2\left(\frac{\pi \lambda}{d}\right)=\frac{d^2-1}{3}.$

Thus the edge $i,j$ contributes 
\[\frac{N(g_{ij}^2-1)}{12w_iw_j}=
\frac{N}{12P}(g_{ij}^2-1)w_kw_\ell. \tag{2.2}\]

The computation of all other simple poles of $F$ give a constant $c$ which is independent of $N$. 

Combining (2.1) and (2.2), the coefficient of $N$ in the total residue is
\[\frac1{12P}\left(S^2+E_2+
\sum_{i<j}(g_{ij}^2-1)w_kw_\ell\right).\]
Since $N=mrS$, 
the coefficient of $m$ in $H_X(m)$ is
$\frac{rS}{12P}
\left(S^2+E_2+B\right)$, where
$B=\sum_{i<j}(g_{ij}^2-1)w_kw_\ell$.

On the other hand,
$\binom{4rm+3}{3}=
\frac{(4rm)^3}{6}+(4rm)^2+\frac{11}{6}(4rm)+1$.
So the coefficient of $m$ is
$\frac{22r}{3}.$ Hence
\[\frac{rS}{12P}(S^2+E_2+B)=\frac{22r}{3}. \]
By using the volume condition 
$S^3=64P$, one has $8(E_2+B)=3S^2$.
Together with
\[3S^2-8E_2=3\sum_iw_i^2-2\sum_{i<j}w_iw_j\
=\sum_{i<j}(w_i-w_j)^2,\]
it follows that 
\[ 8\sum_{i<j}(g_{ij}^2-1)w_kw_\ell=8B=\sum_{i<j}(w_i-w_j)^2\] as desired. 
\end{proof}

\begin{Ex}\label{7abc}
Any 4-tuple $(2a^2, 4b^2, c^2, abc)$ with $2a^2+4b^2+c^2=7abc$ and odd $a,b,c$ has only $A$-type singularities in codimension-two and satisfies the volume equation. There are infinitely many solutions. 

Consider Pell's equation of the form $x^2-33b^2=-8$. One branch of solutions is $x_n+b_n \sqrt{33}=(17+3\sqrt{33})(23+4\sqrt{33})^n$ for $n \ge 0$.
Set $a:=1, b:=b_n, c:=\frac{1}{2}(7b_n+x_n)$, then $(a,b,c)$ is a solution to $2a^2+4b^2+c^2=7abc$.

Notice that $x_{n+1}=23x_n+132b_n$ and $b_{n+1}=4x_n+23b_n$. One can easily check that $(x_n,b_n) \equiv (1,3)$ or $(3,1)$ modulo $4$. It follows that both $b$ and $c$ are odd. 

We set again $a=1$. Then we have 4-tuple $(2, 4b^2, c^2, bc)$ with $2+4b^2+c^2=7bc$ and odd $b,c$. We have to check singularities $\frac{1}{b}(2, c^2)$ and $\frac{1}{c}(2, 4b^2)$ are $A_{b-1}$ and $A_{c-1}$, respectively. By using $2+4b^2+c^2=7bc$, $\frac{1}{b}(2, c^2)$ (resp. $\frac{1}{c}(2, 4b^2)$)  is the same as $\frac{1}{b}(1, -1)$ (resp. $\frac{1}{c}(1, -1)$). Thus every codimension-two singularity is of $A$-type.
\end{Ex}

\begin{Rmk}
The weighted projective spaces in Example~\ref{7abc} except $\PP(1,1,2,4)$ fail (ii) in Lemma~\ref{Hilbert identity}.
Let
\[X=\PP(2a^2,4b^2,c^2,abc),\qquad 2a^2+4b^2+c^2=7abc,\]
where $a,b,c$ are odd. The weights are well formed, so $a,b,c$ are pairwise coprime. Therefore
\[S=2a^2+4b^2+c^2+abc=8abc\]
and $L=\operatorname{lcm}(2a^2,4b^2,c^2,abc)
=4a^2b^2c^2.$ Hence the Cartier index is
\[r=\frac{L}{\gcd(L,S)}=\frac{4a^2b^2c^2}{4abc}=abc.\]

Condition (ii) in Lemma~\ref{Hilbert identity} requires
\[[t^{8ma^2b^2c^2}]\frac1{(1-t^{2a^2})(1-t^{4b^2})(1-t^{c^2})(1-t^{abc})}-
\binom{4mabc+3}{3}
=\frac{m}{6}\left(a^2+2b^2+2c^2-5abc\right)=0.\]
Thus condition (ii) in Lemma~\ref{Hilbert identity} holds precisely when
$a^2+2b^2+2c^2=5abc.$ Combining this with $2a^2+4b^2+c^2=7abc$ shows that only the  solution $\PP(1,1,2,4)$ is possible.
\end{Rmk}

\begin{Ex}
Let $X_n=\PP\!\left(a_n^2,\;2a_na_{n+1},\;a_{n+1}^2,\;32\right)$ be a well-formed weighted projective space satisfying $a_{n+2}=14a_{n+1}-a_n,\ (a_0,a_1)=(1,3)$. Then $X$ satisfies both the volume equation and equation (1.2).

Put $a=a_n,\ b=a_{n+1}$. The recurrence preserves
\[a^2-14ab+b^2+32=0. \tag{*}\]
We have $S=a^2+2ab+b^2+32$ and $P=a^2(2ab)b^2\cdot32=64a^3b^3$. By $(*)$, $S=16ab$, therefore $S^3=64P$.

Since every $a_n$ is odd and $\gcd(a_{n+1},a_{n+2})=\gcd(a_{n+1},14a_{n+1}-a_n)=\gcd(a_{n+1},a_n)=1$.
the nontrivial pairwise gcds are $g_{01}=a,\ g_{12}=b,\ g_{13}=2$, while the other pairwise gcds are $1$.
\[\sum_{i<j}(w_i-w_j)^2-8\sum_{i<j}(g_{ij}^2-1)w_kw_\ell=
\bigl(a^2-14ab+b^2+32\bigr)\bigl(3a^2+38ab+3b^2+96\bigr)=0,\]
since the first factor vanishes by $(*)$. Therefore $X$ satisfies the equation (1.2). But $X$ does not satisfy the condition (iii) in Lemma~\ref{Hilbert identity}. 
Therefore this example gives an  infinite sequence of non-smoothable examples. This example is due to DeVleming.
\end{Ex}

\begin{Rmk}
    Among the list of 23 cases in \cite[Proposition 2.24]{DLT}, 13 cases fail the condition (ii) in Lemma~\ref{Hilbert identity}. The remaining ten cases satisfy condition:
\[\begin{gathered}
(1,1,1,1),\ (1,1,2,4),\ (1,2,9,12),\ (1,4,10,25),\\
(1,6,9,32),\ (1,9,50,60),\ (1,22,32,121),\\
(1,25,65,169),\ (1,50,289,340),\ (2,9,121,132).
\end{gathered}\]
Among these remaining ten cases, two cases, $\PP(1,6,9,32)$
and $\PP(1,22,32,121)$ satisfy Lemma 2.6 but are excluded by Lemma~\ref{smoothing} (iii). The other eight are $\PP^2$-type or $Q$-type.
\end{Rmk}

So by this explicit computation, we get the following.
\begin{Prop}
A well-formed weighted projective space $\PP(w_0, w_1, w_2, w_3)$ admitting a $\QQ$-Gorenstein smoothing to $\PP^3$, with $w_i\le 800$ for each $i$, is the
$\PP^2$-type or the $Q$-type.
\end{Prop}

\section{Finiteness for a fixed gcd}

We begin with the following lemma which shows that there are only finitely many possibilities under some conditions.

\begin{Prop}\label{finite}
Let $X=\PP(1,a,b,c)$ and $\gcd(a,b)=d$ where $d$ is square-free. If $X$ admits a $\QQ$-Gorenstein smoothing to $\PP^3$ then
there are only finitely many possibilities for $X$ with a fixed integer $d$.
\end{Prop}

\begin{proof}
Write
$a=da_0, b=db_0$ with $\gcd(a_0,b_0)=1$.
The volume equation gives $(1+da_0+db_0+c)^3=64d^2a_0b_0c$. Any prime divisor of $d$ is clearly a prime divisor of $(1+c)^3$ hence a prime divisor of $1+c$. Thus,  $d \mid 1+c$ and we may write  ${1+c}={d} \delta$ for some positive integer $\delta$. The smoothing condition in Lemma~\ref{smoothing} (iii) gives two cases.

\medskip
\n Case 1: $d\delta=1+c\in\langle a_0,b_0\rangle$.\\
So $1+c=\alpha a_0+ \beta b_0 $ for some $\alpha, \beta \geq 0$. Let $u:=\alpha+d, v:=\beta+d$. 
Then $u, v \geq d$, and $S=1+da_0+db_0+c=ua_0+ vb_0$. 
The volume equation again gives
\[ (ua_0+ vb_0 )^3=64d^2a_0 b_0c.
\tag{3.1} \]
Since $c<S=ua_0 +vb_0$,
\[ 64d^2=\frac{(ua_0 +vb_0)^3}{a_0b_0c}
> \frac{(ua_0 +vb_0)^2}{a_0 b_0}\geq 4uv. \]
Therefore $uv<16d^2.$ Because $u,v \geq d$, this implies $d\leq u, v < 16d.$ 

Also, equation (3.1) modulo $a_0$ (resp.  $b_0$ ) gives
$a_0\mid v^3$ (resp $b_0\mid u^3$). 
Therefore,
\[a_0\leq v^3 <2^{12}d^3,\qquad
b_0\leq u^3 < 2^{12}d^3. \]
Finally, $c=ua_0+v b_0 < 2^{17}d^4$. Hence, there are only finitely many possibilities for $(a_0,b_0,u,v,c)$, and therefore finitely many weighted projective spaces.

\medskip
\n Case 2: $(d-1)\delta\in\langle a_0,b_0\rangle.$\\
Suppose $(d-1)\delta=\alpha a_0+ \beta b_0.$
Set $u=\alpha+d-1,\ v=\beta+d-1.$
Since $S=1+da_0+db_0+c=d(a_0+b_0+\delta),$ we have
$ua_0+vb_0=(d-1)(a_0+b_0+\delta)=\frac{d-1}{d}S.$
Now the volume equation gives
\[d(ua_0+vb_0)^3=64(d-1)^3a_0b_0c. \tag{3.2}\]
Since $c<S$,
\[64d^2=\frac{S^3}{a_0b_0c}> \frac{S^2}{a_0b_0}
=\frac{d^2}{(d-1)^2}\frac{(ua_0+vb_0)^2}{a_0b_0} \ge \frac{4d^2}{(d-1)^2} uv.\]
We thus obtain $uv < 16(d-1)^2$ and hence $u, v <16(d-1)$.

Similarly, one has $a_0\mid dv^3,\ b_0\mid du^3$ by modulo equation (3.2) with $a_0,b_0$ respectively.
Consequently, $a_0 \le dv^3, b_0\le du^3$ and both are bounded above by $2^{12}d(d-1)^3$. Thus, $c<S< 2^{17}d^2(d-1)^3$. Therefore, finiteness still follows.
\end{proof}

\begin{Prop}
Let $X=\PP(1,a,b,c)$ be a well-formed weighted projective space. Suppose that $\gcd(a,b)=2.$ If $X$ admits a $\QQ$-Gorenstein smoothing to $\PP^3$, then
\[X\cong
\PP(1,1,2,4),\qquad
\PP(1,2,9,12),\qquad\text{or}\qquad
\PP(1,4,10,25).\]
Consequently, $X$ is of $\PP^2$-type or $Q$-type.
\end{Prop}

\begin{proof}
Write $a=2a_0,\ b=2b_0,\ \gcd(a_0,b_0)=1$. Since $X=\PP(1,a,b,c)$ is well-formed and $a,b$ are even, $c$ is odd.
By Lemma~\ref{smoothing}(ii), the generic transverse singularity along the curve determined by the weights $a,b$ is of type $A_1$. Lemma~\ref{smoothing}(iii) therefore gives
$1+c\in \langle a_0,b_0\rangle.$ Thus $1+c=\alpha a_0+\beta b_0$ for some $\alpha,\beta\geq 0$. Set $u=\alpha+2,\ v=\beta+2.$ Then \(u,v\geq 2\), and
$S=ua_0+vb_0.$ The volume equation becomes
\[(ua_0+vb_0)^3=256a_0b_0c,                    
\tag{3.3}\]
where $c=(u-2)a_0+(v-2)b_0-1.$ Since $c<S=ua_0+vb_0$, we have
$256=\frac{(ua_0+vb_0)^3}{a_0b_0c}>\frac{(ua_0+vb_0)^2}{a_0b_0} \geq 4uv.$ Consequently,
$uv<64.$ Reducing (3.3) modulo $a_0$ and $b_0$, respectively, gives
$a_0\mid v^3, \ b_0\mid u^3.$ We now consider the 2-adic valuations $v_2$. Since $c$ is odd, (3.3) gives
\[3 \cdot v_2(ua_0+vb_0)=8+v_2(a_0b_0). \tag{3.4} \]
Because $\gcd(a_0,b_0)=1$, at most one of $a_0,b_0$ is even. They cannot both be odd, since then the right-hand side of (3.4) would be 8, which is not divisible by 3. 
After interchanging $a$ and $b$, we may therefore assume that $a_0$ is odd and $b_0$ is even. Put $\lambda=v_2(b_0)$. Equation (3.4) implies
$\lambda\equiv 1\pmod 3.$ 

First suppose that $\lambda\geq 4$. Then $16\mid b_0$, and (3.4) gives $16\mid ua_0+vb_0$. Since $a_0$ is odd, $16\mid u$. Write
$b_0=16b_1,\ u=16u_1.$ Equations (3.3) and $uv< 64$ become $(u_1a_0+vb_1)^3=a_0b_1c,\ u_1v<4.$ Since $v\geq2$, the only possibilities are $(u_1,v)=(1,2)\ \text{or}\ (1,3).$
If $(u_1,v)=(1,2)$, then $c=14a_0-1.$ Moreover, $a_0\mid v^3=8$, and $a_0$ is odd, so $a_0=1$. We therefore have
$(1+2b_1)^3=13b_1,$ which is impossible for $b_1\geq1$, since
$(1+2b_1)^3-13b_1=8b_1^3+12b_1^2-7b_1+1>0.$ 
If $(u_1,v)=(1,3)$, then $c=14a_0+16b_1-1,$ and hence
$(a_0+3b_1)^3=a_0b_1(14a_0+16b_1-1).$ The difference between the two sides is
$a_0^3-5a_0^2b_1+11a_0b_1^2+a_0b_1+27b_1^3.$ Writing $x=a_0/b_1>0$, this equals
$b_1^3\bigl(x^3-5x^2+11x+27\bigr)+a_0b_1.$ The polynomial
$f(x)=x^3-5x^2+11x+27$ is strictly increasing because $f'(x)=3x^2-10x+11>0$ for every real $x$, and $f(0)=27$. Thus the difference is positive, again a contradiction. 
Therefore $\lambda\geq4$ is impossible.

It remains to consider \(\lambda=1\). Write $b_0=2b_1,$ where $b_1$ is odd. Equation (3.4) gives
$v_2(ua_0+2vb_1)=3.$ In particular, $u$ is even; write $u=2u_1$. Equations (3.3) now become
$(u_1a_0+vb_1)^3=64a_0b_1c.$ Now, $c=(2u_1-2)a_0+(2v-4)b_1-1,$ and $u_1v<32,\ a_0\mid v^3,\ b_1\mid u_1^3.$
Here $a_0,b_1$ are odd and coprime.
The conditions leave only finitely many possibilities:
$1\leq u_1\leq15,\ 2\leq v\leq31,$
with $a_0$ an odd divisor of $v^3$ and $b_1$ an odd divisor of $u_1^3$. This gives exactly the following solutions:
\[\begin{array}{c|c|c|c|c||c|c|c|c|c}
a_0&b_1&u_1&v&c& a_0&b_1&u_1&v&c \\ \hline
1&1&1&3&1& 5&1&1&15&25\\
1&1&2&2&1& 5&1&2&10&25\\
1&3&3&3&9& 5&1&3&5&25 \\
1&3&6&2&9
\end{array}\]
After removing repeated parameterizations and reordering the weights, these give \[\PP(1,1,2,4),\quad
\PP(1,2,9,12),\quad
\PP(1,4,10,25).\] The first and third spaces are of $\PP^2$-type, arising from the Markov triples $(1,1,2)$ and $(1,2,5)$, respectively. The second is of $Q$-type, arising from $(a,b,c)=(1,1,3)$. This proves the proposition.
\end{proof}

\begin{Rmk}\label{algorithm}
Let $X=\PP(1, pa_0, pb_0 ,c), \ \gcd(a_0,b_0)=1,\ \gcd(c,p)=1$, where $p$ is prime, and put
$\delta=\frac{1+c}{p}$. Lemma~\ref{smoothing} shows the smoothing condition along the $A_{p-1}$-curve gives two possibilities:
\[p\delta\in\langle a_0,b_0\rangle
\quad\text{or}\quad
(p-1)\delta\in\langle a_0,b_0\rangle. \]
For each prime $p$, Proposition~\ref{finite} gives an algorithm to find all finite possible cases. 

(i) if $p\delta\in\langle a_0,b_0\rangle$, then it suffices to find all positive integer solutions of  $(a_0,b_0,u,v)$ satisfying: 
$u, v\ge p, uv<16p^2, a_0\mid v^3, b_0\mid u^3,  c=(u-p)a_0+(v-p)b_0-1$, 
$(ua_0+vb_0)^3=64p^2a_0b_0c.$

(ii) if $(p-1)\delta\in\langle a_0,b_0\rangle$, then  find all positive integer solutions of  $(a_0,b_0,u,v)$ satisfying: \\
$p-1\le u,v <16(p-1), a_0\mid pv^3, b_0\mid pu^3$, $c=\frac{p}{p-1}\bigl(a_0(u-p+1)+b_0(v-p+1)\bigr)-1, \ p(ua_0+vb_0)^3=64(p-1)^3a_0b_0c.$
\end{Rmk}
 
We compute for prime \(p\le100\) from Remark~\ref{algorithm} (i), then we get the following list:

\[
\begin{array}{c|l|l}
p	& \text{remaining weight systems}	 & \text{type} \\ \hline
2 	& (1,1,2,4), \ (1,4,10,25) &	\PP^2 \\
&	(1,2,9,12) &	Q\\ \hline
3 &	(1,2,9,12), \ (1,9,50,60)	& Q \\ \hline
5 & (1,4,10,25), \ (1,25,65,169) &	 \PP^2\\ \hline
13 & (1,25,65,169), \ (1,169,442,1156) & \PP^2\\ \hline
17 & (1,50,289,340), \ (1,289,1682,1972) &	Q \\ \hline
89 & (1,1156,3026,7921), \ (1,7921,20737,54289)  &	\PP^2\\
\end{array}\]

If we compute for prime $p\le100$ from Remark~\ref{algorithm} (ii), then we get the following  list:
$\PP(1,4,10,25),\ \PP(1,25,65,169),\ \PP(1,169,442,1156), \ \PP(1, 7921, 20737, 54289)$, which is a subset of the above table.

Hence the second case introduces no additional smoothable weighted projective space up to $p\le 100$.
Every other prime $p\le100$ gives no candidate. So by explicit computation, we have
\begin{Prop}	
Let $X=\PP(1,a,b,c)$ be well-formed, with $\gcd(a,b)=p$, where $p\le100$ is prime. Assume that $X$ admits a $\QQ$-Gorenstein smoothing to $\PP^3$. Then $X$ is of $\PP^2$-type or $Q$-type.
\end{Prop}

\section{Main theorem}

\begin{Thm}\label{main}
Let $X=\PP(1,a,b,c)$ be a well-formed weighted projective space. Suppose $d:=\gcd(a,b)$ and $a=d^2$. Assume every codimension-two singularity of $X$ is $A$-type. If $X$ admits a $\QQ$-Gorenstein smoothing to $\PP^3$, then $X$ is of $\PP^2$-type or $Q$-type.
\end{Thm}

\begin{proof}
Write $b=db_0$. We have $\gcd(d,b_0)=1$ and 
$\gcd(d,c)=1$. Write $S=1+d^2+db_0+c$. The volume equation gives $S^3=64d^3b_0c$. So there is an integer $m>0$ such that $S=4dm$ and $b_0c=m^3$. Let $e=\gcd(b_0,c).$

\n {\bf Claim 1.}
We get the following equation 
\[b_0^2-4mb_0+6m^2-2c-e^2=0. \tag{1}\]

\begin{proof}[Proof of Claim 1]
Lemma~\ref{linear} shows that
\[ \sum_{i<j}(w_i-w_j)^2=
8\left((d^2-1)c+(e^2-1)d^2\right).
\tag{2} \]
Use the identity
$\sum_{i<j}(w_i-w_j)^2
=3S^2-8\sum_{i<j}w_iw_j$, Therefore (2) becomes
\[\begin{aligned}
3S^2
&-8\left(
d^2+db_0+c+d^3b_0+d^2c+db_0c
\right)\\
&=8\left((d^2-1)c+(e^2-1)d^2\right).
\end{aligned} \]

Since $S=4dm$, it follows that
\[6d^2m^2-\left(d^2+db_0+c+d^3b_0+d^2c+db_0c\right)
=(d^2-1)c+(e^2-1)d^2.\] 
Hence $d^2(6m^2-2c-e^2)-db_0(1+d^2+c)=0.$
By the relation
$1+d^2+db_0+c=4dm$, we have 
\[d^2(6m^2-2c-e^2)
-db_0\cdot d(4m-b_0) =d^2 (6m^2-2c-e^2-4mb_0+b_0^2)=0. \]
Since $d > 0$, the Claim 1 is verified. 
\end{proof}

\n {\bf Claim 2.} We have $e|m$. 
\begin{proof}[Proof of Claim 2]
Let $p$ be a prime and set $\beta=v_p(b_0), \gamma=v_p(c), \mu=v_p(m)$.
Since $b_0c=m^3$,
we have $\beta+\gamma=3\mu$. Also, because $e=\gcd(b_0,c)$,
$\epsilon:=v_p(e)=\min\{\beta,\gamma\}$. We want to prove $\epsilon \le \mu.$ 
Suppose on the contrary that $\epsilon> \mu$.
Then both $\beta>\mu$ and $\gamma>\mu$. From $\beta+\gamma=3\mu$, we obtain $\gamma=3\mu-\beta<2\mu$. 

\medskip
\n
Case 1: $p$ is odd.\\
In this case, one can easily see that 
$v_p(6m^2-e^2-4mb_0+b_0^2) \ge 2 \mu$. However, $v_p(2c)=\gamma < 2\mu$. This is the required contradiction. 

\medskip 
\n
Case 2: $p=2$.\\
Now $v_2(2c)=\gamma+1$. One can also easily see that 
$v_2(6m^2-e^2-4mb_0+b_0^2) \ge 2 \mu+1$. This is again a contradiction.

 Hence, we get $v_p(e)  \le v_p(m) $ for any prime factor of $m$. It follows that $e |m$.
\end{proof}

We write $m=en$, $b_0=er_0$, $c=es$ with 
$\gcd(r_0,s)=1$. By Claim 1, $e^2|2c$. Hence we may write $2s=et$. 
 Then $r_0t=2n^3$ and 
\[r_0^2-4nr_0+6n^2-t-1=0. \tag{3}  \]
We have $\gcd(r,t)=1$;
for odd primes this follows form $2s=et$ and $\gcd(r_0, s)=1$. The prime 2 comes from the equation (3). Reducing (3) modulo 2 shows that if $r_0$ is even then $t$ is odd. 

Consequently, either
\[ (r_0=x^3,\ t=2y^3,\ n=xy), \quad \text{or} \quad (r_0=2x^3,\ t=y^3,\  n=xy) \tag{4} \]
for some positive integers $x$ and $y$. Let
\[ F(U,V)=U^3-4U^2V+6UV^2-2V^3. \]

Substituting the two cases in (4) into (3) gives, respectively
\[ F(x^2,y)=1 
\quad\text{or}\quad
F(2x^2,y)=2. \]

\n {\bf Claim 3.} The positive integral solutions of $F(U,V)=1$ are $(U,V)=(1,1),(1,2)$, and the unique positive integral solution of $F(U,V)=2$ is
$(U,V)=(2,1)$.
\begin{proof}[Proof of Claim 3.]
The homogeneous polynomial $F(U,V)$ is an irreducible binary cubic form over $\QQ$ of degree $\ge 3$. So these are Thue equations. By Thue’s theorem (cf. \cite{Thue}), there are only finitely many integral solutions to the equation $F(U,V)=c$ for any non-zero integer $c$.
An exact computation using Magma’s {\it Thue and Solutions} functions gives all integral solutions; filtering for $U,V>0$ gives the above lists. The computation uses the Bilu–Hanrot Thue-equation algorithm \cite{BH-T}.
\end{proof}

\noindent Case 1. $b_0=m,\ c=m^2$.\\
The Markov equation $1+d^2+m^2=3dm$ follows from $1+d^2+db_0+c=4dm$, which is the volume equation. Therefore
\[X=\PP(1,d^2,dm,m^2),\]which is of $\PP^2$-type.

\noindent Case 2. $b_0=2m,\ c=m^2/2$.\\
The gcd requirement $e=m$ gives
$m=2x$ where $x$ is odd.
The Markov equation $1+d^2+2x^2=4dx$ follows from $1+d^2+db_0+c=4dm$, which is the volume equation. Moreover,
\[X=\PP(1,d^2,4dx,2x^2).\]
After reordering,
$X=\PP(2x^2,1,d^2,4xd),$ which is $Q$-type.

\noindent Case 3. $b_0=m/2,\ c=2m^2$.\\
Write $m=2q$. Then $b_0=q$ and $c=8q^2$. Along the gcd-\(q\) curve determined by the weights \(dq,8q^2\), the smoothing-section condition gives
two alternatives by Lemma~\ref{smoothing} (iii), $1+d^2=du+8qv, \ u,v\ge0$, or $(q-1)(7d-8q)=du+8qv,$  for $u,v\ge0$. This numerical branch already satisfies the Hilbert identity. 
We now show that neither alternative in Lemma~\ref{smoothing} (iii) can occur.

We consider the first case. If $q=1$ then $d^2-7d+9=0$, which has no integral solution.
To prove $q\mid u$, we explicitly use $1+d^2=q(7d-8q$), so that $q\mid1+d^2$. Reducing $1+d^2=du+8qv$ gives $q\mid u$. Write $u=qu_0$. Dividing by $q$,
\[(7-u_0)d=8(q+v). \tag{5}\] The right-hand side is positive, so $0\le u_0\le6$. Since $\gcd(d,8)=1$, equation (5) implies $8\mid 7-u_0$,
which is impossible because $1\le7-u_0\le7$. Thus it has no global smoothing direction.

Suppose
$(q-1)(7d-8q)=du+8qv$, for $u,v\ge 0$. Reducing modulo $q$ gives
$-7d\equiv du\pmod q.$ Since \(\gcd(d,q)=1\), $u\equiv-7\pmod q.$ Hence
$u=qk-7$ for some integer $k\ge1$. Substituting this into the equation $(q-1)(7d-8q)=du+8qv$ gives
$(q-1)(7d-8q)=d(qk-7)+8qv$.
After expansion, cancellation, and division by $q$, we obtain
\[(7-k)d=8(q+v-1).\tag{6}\]The right-hand side is positive, so $1\le k\le6$. Moreover, $\gcd(d,8)=1$, because $\gcd(d,8q^2)=1$. Equation (6) therefore implies $8\mid7-k.$ But $1\le7-k\le6$, which is impossible.
\end{proof}

\begin{Rmk}
In the paper, we did not use compatibility of the local deformations at the coordinate points. The deformation-gluing results of DeVleming–Li–Torres \cite{DLT} provide a framework for studying exactly these codimension-three compatibility conditions, but not the desired numerical classification.
\end{Rmk}

\begin{Rmk}
Let $X=\PP(1, a, b, c)$ and $\gcd(a,b)=d, \ \gcd(b,c)=e,\ \gcd(a,c)=1$. Write
$a=da_0,\ b=deb_0,\ c=ec_0$. The gcd conditions imply
$\gcd(a_0,eb_0)=\gcd(db_0,c_0)=\gcd(da_0,ec_0)=1$. In particular,
$\gcd(d,e)=1,\ \gcd(a_0,b_0)=\gcd(b_0,c_0)=\gcd(a_0,c_0)=1.$

Put $S=1+da_0+deb_0+ec_0.$ The volume equation gives $S^3=64d^2e^2a_0b_0c_0.$ Consequently,
$S=4m, \ \ m^3=d^2 e^2 a_0 b_0 c_0$ for some positive integer $m$. Then as in the proof of Theorem~\ref{main}, Hilbert identity gives the equation
$b_0^2-4\rho b_0+6\rho^2-\alpha-\gamma-\alpha\gamma=0,$ where
$\rho=\frac{m}{de},\
\alpha=\frac{a_0}{d},\
\gamma=\frac{c_0}{e}$ are positive rational numbers.
When $a=d^2$, one has $a_0=d$, hence $\alpha=1$. That is precisely the specialization which makes the argument in Theorem~\ref{main} reduce to cubic Thue equations. In the general case, both $\alpha$ and $\gamma$ vary.
If one could prove $m=de$, the resulting equation would be substantially more restrictive and may reduce the problem to the two expected families. At present, however, we do not know how to prove $m=de$.
\end{Rmk}


\end{document}